\documentclass[12 pt, reqno]{amsart}
\usepackage[utf8]{inputenc}
\usepackage{cite}
\usepackage{amssymb, amsfonts, amsthm,amsmath,float}
\usepackage{graphicx,a4wide, color}
\numberwithin{equation}{section}
\newtheorem{theorem}{Theorem}[section]
\newtheorem{lemma}[theorem]{Lemma}

\newtheorem{proposition}[theorem]{Proposition}
\theoremstyle{definition}
\newtheorem{definition}[theorem]{Definition}
\newtheorem{remark}[theorem]{Remark}
\newtheorem{example}[theorem]{Example}

\begin{document}
\title[Orthonormal Wavelet System in $\ell^2 ({\mathbb{Z}}^2_N)$]
{Orthonormal Wavelet System in $\ell^2 ({\mathbb{Z}}^2_N)$}
\author{Anupam Gumber  and Niraj K. Shukla} 

 \thanks{Address: Discipline of Mathematics,
	Indian Institute of Technology Indore, Simrol, 
	Indore-452 020, India. \\
   E-mail: anupamgumber.iiti@gmail.com, o.nirajshukla@gmail.com}
\begin{abstract}
	 Using the group theoretic approach based on the set of digits, we first investigate 
		a finite collection of functions in $\ell^2 ({\mathbb{Z}}^2_N)$
		that satisfies  some  localization properties  in a region of the time-frequency plane. The digits are  associated with an invertible (expansive/non-expansive)  matrix having integer entries. Next,  we study and characterize an orthonormal wavelet system   for $\ell^2 ({\mathbb{Z}}^2_N)$.    In addition,   some results connecting the uncertainty principle with functions 
		that generate the orthonormal wavelet system having time-frequency localization properties are  obtained. 
\end{abstract}
\maketitle

\section{Introduction}\label{sec1}
A time-frequency localized basis plays an important  role in extracting both time as well as  frequency 
information of a given signal, and the uncertainty principle helps us to understand how much local information
in time and frequency we can extract by using the above mentioned basis. Such basis has many applications in the real life 
problems, for example, to look 
closely at a potential tumor in medical image processing, in radar or sonar imaging, for compressing video images in 
video image analysis, for the recovery of lost part of the signals, etc. Due to this, many researchers are attracted towards  
the problem of finding  good bases having
both time as well as frequency localization properties (e.g. \cite{F, PBV,RT,V}). Furthermore, in this day and age of computers, processing can be done
only when the signal can be stored in memory. Therefore, the importance of discrete and finite signals can not be ignored. 
 
Wavelets are the latest and most successful tools to extract information from many different kinds of data including  but certainly 
not limited to  audio signals and images. 
Our main goal is to study   orthonormal wavelet systems 
having time-frequency localization properties for multidimensional setup. In particular,   we want to find 
 conditions on the sets $I_1, I_2$ and $A \in {GL}(2,\mathbb{R})$,  and to characterize 
$\Phi=\{\varphi_p\}_{p \,\in I_2} \subset \ell^2 ({\mathbb{Z}}^2_N)$ such that  the  
 collection $\mathfrak{B}(\Phi)$ (with distinct elements) 
defined by  
 \begin{equation}\label{eq 1.1}
\mathfrak{B}(\Phi)  := \{ T_{Ak} \varphi_{p}: k \in I_1 \subseteq {\mathbb{Z}}^2_N ,\, p \in I_2 \subset \mathbb{N} \}  
 \end{equation}
 forms an orthonormal basis for $\ell^2 ({\mathbb{Z}}^2_N)$ having time-frequency 
 localization properties, where for each $k \in {\mathbb{Z}}^2_N$ and $f \in \ell^2({\mathbb{Z}}^2_N)$, 
the \textit{translation operator} $T_k$    on $\ell^2({\mathbb{Z}}^2_N)$ is  defined by $T_{k}(f)(m)= f(m-k),\,\, \mbox{for all}
\,\,m \in {\mathbb{Z}}^2_N.$ For $N \in \mathbb N$,   the space $\ell^2 ({\mathbb{Z}}^2_N)$ denotes the collection  of all
$N\times N$ complex matrices by identifying 
$f \in \ell^2 ({\mathbb{Z}}^2_N)$ with 
the $N\times N$ matrix
 $ 
{{\begin{pmatrix}
f({n}_1,{n}_2)^t
\end{pmatrix}}}_{\substack{
   {n}_1,{n}_2 \in {\mathbb{Z}}_N 
  }}.
  $
Note that  it is an $N^2$-dimensional Hilbert space with usual inner product 
$$
 \langle f,g \rangle=\mbox{trace}(g^{*}f)= 
 \sum \limits_{(n_1,n_2)^t \,\in \, {\mathbb{Z}}^2_N} f(n_1,n_2)^t \overline {g(n_1,n_2)^t}, \  \ \mbox{for all}\  f, g  \in \ell^2 ({\mathbb{Z}}^2_N),
 $$
  where 
$g^{*}$  and $\mathbb{Z}_N:=\{0,1,2,\ldots,N-1 \}$ denote the 
conjugate transpose of $g$ and a group of integers under addition modulo $N$, respectively,  and  
${\mathbb{Z}}^2_N=\mathbb{Z}_N \times \mathbb{Z}_N=\{(m, n)^t:={m \choose n}: m, n \in \mathbb Z_N\}$.

By a \textit{time-frequency localized basis}, we mean that 
every vector in the basis is time localized as well as 
frequency localized.  An   $f\in \ell^2 ({\mathbb{Z}}^2_N)$ is \textit{time localized}  near $n_0\in \mathbb Z_N^2$ if most of the  components $f(n)$ of $f$ are $0$ or relatively small, except for a few values of $n$ close to $n_0$ and it is said to be  \textit{frequency localized} near $n_0$ if most of the components of discrete Fourier transform  of $f$ are $0$ or relatively small, except for a few values of $n$ close to $n_0$.   At this juncture, it is pertinent to note  that     the standard 
orthonormal  basis   for $\ell^2 ({\mathbb{Z}}^2_N)$ is  time localized but not frequency localized, while  its Fourier basis   
is frequency localized but not time localized.  
 In the last two decades a lot of research on time-frequency analysis  has been done by several authors for the various spaces, 
namely, finite and infinite  abelian groups, Euclidean spaces, etc. (e.g. \cite{F, GJ, GLT, GT,  H, KPR, RT, V, W}). 

In order to compare $\mathfrak{B}(\Phi)$ defined by (\ref{eq 1.1})  with the classical notion of orthonormal multiwavelet \cite{W}, 
let us consider $\Psi=\{ \psi_1, \psi_2, \ldots , \psi_L\} \subset L^{2}({\mathbb{R}}^n)$ 
with the dilation  matrix $A_0$ and translation ${\mathbb{Z}}^n$, which provides an orthonormal basis $\mathcal{A}(\Psi)$ 
for $L^{2}({\mathbb{R}}^n)$, where $$\mathcal{A}(\Psi):= \{ {\psi^l_{j,k}}:={|\det(A_0)|}^{j/2}\psi_l(A^{j}_{0}.-k): j \in \mathbb{Z},
k \in {\mathbb{Z}}^n ~\mbox{and}~ \psi_l \in \Psi\}.$$ In view of the system  $\mathcal{A}(\Psi)$, note that    elements of 
$\mathfrak{B}(\Phi)$ can be written  as
${\varphi^p_{0,Ak}}=T_{Ak} \varphi_p$, where $p \in I_2, k \in I_1.$ Further, we remark that if $\mathfrak {B}(\Phi)$ is an orthonormal basis for $\ell^2({\mathbb{Z}}^2_N)$, then one cannot construct an orthogonal set
by adding elements into $\mathfrak {B}(\Phi)$. This, in turn,  implies that there is no need to use dilation operator in the definition of an orthonormal wavelet system (ONWS) in $\ell^2({\mathbb{Z}}^2_N)$. Thus,
it leads us to conclude that we cannot use the natural definition of orthonormal multiwavelet of  $L^2(\mathbb R^n)$ in our case. 

 Further, in this paper,  
we study Heisenberg uncertainty principle in terms of the pair of representations of a given signal  by coupling 
ONWS with the standard basis and Fourier basis. This study  can  be useful to provide uniqueness 
properties of the sparse representation of the signal. The theory of   uncertainty principle related to the pair of bases 
and its applications  has been developed by many authors (e.g. \cite{DH,EB, GJ,K,KPR, RT}).  

The remainder of the  article is organized as follows:  In Section 2, 
 we make a necessary background to 
investigate a finite collection of functions in $\ell^2 ({\mathbb{Z}}^2_N)$
that satisfies  some  localization properties  in a region of the time-frequency plane  
using the group theoretic approach based on the set of digits. Using these functions, we first introduce an orthonormal wavelet system 
in $\ell^2({\mathbb{Z}}^2_N)$ and then provide a characterization for its generators. 
 Section $3$ presents some results on uncertainty principle corresponding to this 
orthonormal wavelet system. The results obtained in Sections 2 and 3 for $\mathbb Z_N^2$ can be generalized for $\mathbb Z_N^M$, $M \in \mathbb N$.

   Throughout the paper,   $\mathcal{M}(2,X)$ and ${GL}(2,X)$ denote  the collection of all $2\times2$ matrices and 
 collection of all $2\times2$ invertible matrices over $X$, respectively, where   $X$ is used as $\mathbb{Z}_N ~\mbox{or}~ 
 \mathbb{Z}$ (the set of integers)  or $\mathbb{R}$ (the set of real numbers). The notation $|\cdot|$ denotes the cardinality of a set, or, the absolute value of a complex number. For $A=\Big({\begin{matrix}
 	~a&b~ \\
 	~ c&d ~
 	\end{matrix}}\Big) \in \mathcal{M}(2,\mathbb{Z})$  and $I \subseteq \mathbb Z_N^2$,  
 \begin{align*}
 AI  
 	 :=\Big\{\Big({\begin{matrix}
 			~a&b~ \\
 			~ c&d ~
 		\end{matrix}}\Big)\Big({\begin{matrix}
 		~n_1~ \\
 		~ n_2 ~
 	\end{matrix}}\Big) \, modN = \Big({\begin{matrix}
 			~ an_1 \oplus bn_2~ \\
 			~cn_1 \oplus dn_2 ~
 		\end{matrix}}\Big):(n_1,n_2)^t \in I  \Big\}\subset \mathbb{Z}_N^2,
 	\end{align*}
 	$ \mbox{where}~ \oplus ~\mbox{denotes addition modulo}~ N.$ 


\section{Group theoretic approach for ONWS in $\ell^2({\mathbb{Z}}^2_N)$ }\label{sec2}
 Our main motive in this section is to answer the question imposed for the system $\mathfrak B(\Phi)$ defined in (1.1). 
 For this, we need to recall and establish some group theoretic results. 
From (\ref{eq 1.1}), it is clear that   $A {\mathbb{Z}}^2 \subseteq{\mathbb{Z}}^2,$  and $|\mathfrak{B}(\Phi)|= |AI_1||I_2|\leq |I_1||I_2|$ . Further, we note that the system $\mathfrak{B}(\Phi)$ should have $N^2$ distinct elements   to become  an orthonormal basis for 
$\ell^2 ({\mathbb{Z}}^2_N)$. Hence, we have
 $N^2=|\mathfrak{B}(\Phi)|= |AI_1||I_2|$, which implies that $|AI_1|= N^2/|I_2|.$ To find $I_1$ 
 and $I_2$, firstly we are interested in computing $|AI_1|$, where $I_1 \subseteq {\mathbb{Z}}^2_N$. 
 Clearly, $A {\mathbb{Z}}^2_N \subseteq{\mathbb{Z}}^2_N$. 
Now, by using the fact that
  ${\mathbb{Z}}^2_N$ is an abelian group, we   conclude that $A {\mathbb{Z}}^2_N$ is a normal subgroup of
    ${\mathbb{Z}}^2_N$. The motivation for considering subgroups of ${\mathbb{Z}}^2_N$ of this type comes from 
    the following theorem:
 \begin{theorem} \label{thm 2.1}
  Every subgroup of $ {\mathbb{Z}}^2$ is of the form $A{\mathbb{Z}}^2$ for some 
  $A \in \mathcal{M}(2,\mathbb{Z}).$
 \end{theorem}
 \begin{proof}
 We start by claiming that every subgroup of $ {\mathbb{Z}}^2$ is generated 
 by at most two generators. This means, for some integer $n\geq 1$, we have to show that 
 a subgroup of ${\mathbb{Z}}^2$, which is generated by $n$ generators is actually generated by at most two.
 For this, it is enough to assume three elements $(a,b)^t,(c,d)^t,(e,f)^t \in {\mathbb{Z}}^2$, which are
 supposed to generate a subgroup $H$ of ${\mathbb{Z}}^2$ and we need to show that one of them is redundant. 
 For this, let $a,c\neq 0$. Then, we choose $x,y\in \mathbb{Z}$ such that $x(a,b)^t+y(c,d)^t=(0,b')^t$ for some $b' \in \mathbb{Z}$.
Clearly, $(0,b')^t,(c,d)^t$ and $(e,f)^t$ will 
 generate $H$. If $b'=0$, then the result follows. Suppose $b'\neq 0$. By proceeding in the similar way,
  choose $x',y' \in \mathbb{Z}$ such that $x'(c,d)^t+y'(e,f)^t=(e',0)^t$ for some $e' \in \mathbb{Z}$. Then, $(0,b')^t,(c,d)^t$ 
  and $(e',0)^t$ will generate $H$, and hence, we can select   $m,n,p \in \mathbb{Z}$ satisfying the equation 
  $m(0,b')^t+ n(c,d)^t+ p(e',0)^t=(0,0)^t$, for which one choice is
 $m=-e'd,n=e'b'$, and $p=-b'c$. Therefore,    $m(0,b')^t+ n(c,d)^t + p(e',0)^t,(c,d)^t $ and $(e',0)^t$ are also generators of
 $H$, which implies that $H$ is generated by $(c,d)^t$ and $(e',0)^t$. Combining the facts discussed above, we get
 $
  H=\langle{(a,b)^t, (c,d)^t, (e,f)^t}\rangle =\langle{(c,d)^t,(e',0)^t}\rangle.
  $
  Further note that, if $(c,d)^t$ and $(e',0)^t$ are linearly independent, then $H$ is generated by two generators, and hence
  we can write
\begin{equation*}
  H=\langle{(c,d)^t,(e',0)^t}\rangle
  = \Big\{\Big({\begin{matrix}
~c&e'~ \\
~ d&0 ~
\end{matrix}}\Big)\Big({\begin{matrix}
m \\
n 
\end{matrix}}\Big) : m,n \in \mathbb{Z} \Big \}
=  \Big({\begin{matrix}
~c&e'~ \\
~d &0 ~
\end{matrix}}\Big){\mathbb{Z}}^2=
\Big({\begin{matrix}
~e'&c~ \\
~0 &d ~
\end{matrix}}\Big){\mathbb{Z}}^2.
  \end{equation*} 
Otherwise, $H$ will be generated by only one generator say $(c,d)^t$, and then it can be written as 
  \begin{equation*}
  H=\langle{(c,d)^t}\rangle
  = \{(cm,dm)^t : m \in \mathbb{Z} \}
  = \Big\{\Big({\begin{matrix}
~c&0~ \\
~ d&0 ~
\end{matrix}}\Big)\Big({\begin{matrix}
m \\
n 
\end{matrix}}\Big) : m,n \in \mathbb{Z} \Big \}
= \Big({\begin{matrix}
~c&0~ \\
~ d&0 ~
\end{matrix}}\Big){\mathbb{Z}}^2=\Big({\begin{matrix}
~0&c~ \\
~ 0&d ~
\end{matrix}}\Big){\mathbb{Z}}^2.
  \end{equation*}
Hence, we conclude the result from the fact that in each case, there exists $A \in  \mathcal{M}(2,\mathbb{Z})$ such that 
$H=A{\mathbb{Z}}^2.$
 \end{proof}
For a $2 \times 2$ expansive matrix $A$ having integer entries, it is well known that the order of group 
$\displaystyle \frac{{\mathbb{Z}}^2}{A {\mathbb{Z}}^2}$ is $|\det(A)|$ \cite[Proposition 5.5]{W} while  the result is also 
true for any invertible matrix with integer entries. By an \textit{expansive matrix}, we mean that all of its eigenvalues $\lambda$ satisfy $|\lambda|>1$.  
Now, we are interested in finding the order of group $A {\mathbb{Z}}^2_N$. Therefore, 
we have the following result:
 \begin{theorem} \label{thm}
  Let $ N \in \mathbb{N}$ and $A \in {GL}(2, \mathbb{R})$ such that  $A {\mathbb{Z}}^2 \subseteq{\mathbb{Z}}^2
   $, and let $B \in {GL}(2, \mathbb{Z})$, where $B=NA^{-1}$. Then, the determinant $\det(A)$ of $A$ divides $N^2$, and hence 
   the order of  groups $A {\mathbb{Z}}^2_N$ and 
   $\displaystyle \frac{{\mathbb{Z}}^2_N}{A{\mathbb{Z}}^2_N}$ is given by
   $\displaystyle |A {\mathbb{Z}}^2_N|=|\det(B)|= \frac{N^2}{|\det(A)|}~\mbox{and}~ 
    \Big|\displaystyle \frac{{\mathbb{Z}}^2_N}{A{\mathbb{Z}}^2_N}\Big|
   =|\det (A)|$, respectively.
  \end{theorem}
\begin{proof}
 Let $A=\Big({\begin{matrix}
~a&b~ \\
~c&d ~
\end{matrix}}\Big)\in {GL}(2, \mathbb{R})$ be such that $a,b,c,d \in \mathbb{Z}$. Then,
$B=NA^{-1} \in {GL}(2, \mathbb{Z})$
 is given by $\Bigg({\begin{matrix}
~\frac{Nd}{\det(A)}& \frac{-Nb}{\det(A)}~ \\
~ \frac{-Nc}{\det(A)}&  \frac{Na}{\det(A)} ~
\end{matrix}}\Bigg).$ 
Since 
$\displaystyle \frac{N^2}{\det(A)}= \frac{Na}{\det(A)} \times \frac{Nd}{\det(A)}-
\frac{Nb}{\det(A)} \times 
\frac{Nc}{\det(A)}$ 
 is an integer, the  $\det(A)$ divides $N^2.$ 
Next, we have to show that
$\displaystyle |A {\mathbb{Z}}^2_N|=|\det(B)|=\frac{N^2}{|\det(A)|}.$
For this, it is enough to see isomorphism between groups $A{\mathbb{Z}}^2_N$ and $\displaystyle \frac{{\mathbb{Z}}^2}{B
{\mathbb{Z}}^2}$, that means, $A{\mathbb{Z}}^2_N \cong \displaystyle \frac{{\mathbb{Z}}^2}{B{\mathbb{Z}}^2}$, which implies
$|A{\mathbb{Z}}^2_N| = \displaystyle \Big |\frac{{\mathbb{Z}}^2}{B
{\mathbb{Z}}^2}\Big|=|\det(B)|=\frac{N^2}{|\det(A)|},$ and hence $
    \Big|\displaystyle \frac{{\mathbb{Z}}^2_N}{A{\mathbb{Z}}^2_N}\Big|=
     \displaystyle \frac{|{\mathbb{Z}}^2_N|}{|A{\mathbb{Z}}^2_N|}= 
     |\det (A)|.$
 Now, in order to prove the above claim, consider a map
$  f: {\mathbb{Z}}^2 \rightarrow A{\mathbb{Z}}^2_N$
 defined for all $(n_1,n_2)^t \in {\mathbb{Z}}^2$ by 
 $f(n_1,n_2)^t=A\Big({\begin{matrix}
r_1 \\
r_2 
\end{matrix}}\Big)\, mod\,N,\,\mbox{where}~n_i\equiv r_i\,mod\,N,~\mbox{for}~i=1,2.$
 Clearly, $f$ is well defined and further, the homomorphism of $f$ follows by noting that 
 for all $ (n_1,n_2)^t,(m_1,m_2)^t \in {\mathbb{Z}}^2$, we can write
 $n_i\equiv r_i\,mod\,N, m_j\equiv R_j\,mod\,N, \mbox{and}~ 
(n_i + m_j)\equiv R_{ij} \,mod\,N$, for $i,j=1,2$, and hence
\begin{align*}
 f((n_1,n_2)^t+(m_1,m_2)^t)
  &=A\Big({\begin{matrix}
R_{11} \\
R_{22}
\end{matrix}}\Big)\, mod\,N = A\Big({\begin{matrix}
r_1 \oplus R_1 \\
r_2 \oplus R_2
\end{matrix}}\Big)\, mod\,N\\
&= A\Big({\begin{matrix}
\Big({\begin{matrix}
r_1 \\
r_2 
\end{matrix}}\Big) \oplus \Big({\begin{matrix}
 R_1 \\
 R_2
\end{matrix}}\Big)
\end{matrix}}\Big)\, mod\,N = f(n_1,n_2)^t\oplus f(m_1,m_2)^t.
 \end{align*}
 Next, the map $f$ is onto and its kernel Ker($f$) is given by
\begin{align*}
~\mbox{Ker}(f)&=\Big \{ (m_1,m_2)^t \in {\mathbb{Z}}^2: f(m_1,m_2)^t=  A\Big({\begin{matrix}
 R_1 \\
 R_2
\end{matrix}}\Big)\,mod\,N=\Big({\begin{matrix}
 0 \\
 0
\end{matrix}}\Big)\Big \}\\
&= \Big \{ (m_1,m_2)^t \in {\mathbb{Z}}^2:~\mbox{for}~ q_1,q_2 \in \mathbb{Z},A\Big({\begin{matrix}
 m_1-Nq_1 \\
 m_2-Nq_2
\end{matrix}}\Big)\,mod\,N=\Big({\begin{matrix}
 0 \\
 0
\end{matrix}}\Big)\Big \}\\
             &= \Big \{ (m_1,m_2)^t \in {\mathbb{Z}}^2: A\Big({\begin{matrix} m_1 \\ m_2 \end{matrix}}\Big) \in N\mathbb{Z}\times N\mathbb{Z}=NI{\mathbb{Z}}^2 \Big \}\\
              &= \Big \{ (m_1,m_2)^t \in {\mathbb{Z}}^2: (m_1,m_2)^t \in NA^{-1}{\mathbb{Z}}^2=
              B{\mathbb{Z}}^2 \Big\}\\
              &= B{\mathbb{Z}}^2
              ,~\mbox{where}~ I  
             ~\mbox{is a}~ 2\times 2 ~\mbox{identity matrix}.
\end{align*}
Therefore, by fundamental theorem of group homomorphism,
$
A{\mathbb{Z}}^2_N \cong \displaystyle \frac{{\mathbb{Z}}^2}{B
{\mathbb{Z}}^2}
$. 
\end{proof}
\begin{remark}\label{re 2.5}
 If $ A \in {GL}(2, \mathbb{R})$ is a matrix such that $A {\mathbb{Z}}^2 \subseteq{\mathbb{Z}}^2$ and 
 $\det(A)$ divides $N^2,$ then 
 the matrix $N {A}^{-1}$ need not be a member of ${GL}(2, \mathbb{Z}).$ 
For example, let $A=\Big({\begin{matrix}
~3&1~ \\
~1&3 ~
\end{matrix}}\Big)$ and $N=4$.
Then $\det(A)~\mbox{divides}~N^2$, but 
$NA^{-1}=\displaystyle \frac{4}{8}\Big({\begin{matrix}
~3&-1~ \\
~-1&3 ~
\end{matrix}}\Big)=\Big({\begin{matrix}
~3/2&-1/2~ \\
~-1/2&3/2 ~
\end{matrix}}\Big) \notin {GL}(2, \mathbb{Z}).$
\end{remark}
   From Theorem \ref{thm}, we have 
  $\Big|\displaystyle \frac{{\mathbb{Z}}^2_N}{B{\mathbb{Z}}^2_N}\Big| =|\det (B)| = \frac{N^2}{|\det (A)|}=r,~\mbox{(say)}$.
  In the rest of the paper, by $\{\overline{\beta}_i\}_{i=1}^{r}$ we denote the distinct coset representatives of
   $ \displaystyle \frac{{\mathbb{Z}}^2_N}{B {\mathbb{Z}}^2_N}$, where for
   $1\leq i \leq r,~\mbox{}~\overline{\beta}_i= {\beta}_i + B {\mathbb{Z}}^2_N$ such that
   ${\beta}_i\in 
   \mathfrak{D}.$ 
   Here, the set $\mathfrak{D}\subseteq ({\mathbb{Z}}^2_N\setminus B{\mathbb{Z}}^2_N) \cup \{(0,0)^t\}$ is termed as
   {\it{the 
   set of digits}} corresponding to the cosets of $B {\mathbb{Z}}^2_N ~\mbox{in}~{\mathbb{Z}}^2_N.$ Clearly,
   for $1 \leq i\neq j \leq r, ~\mbox{we have}~{\beta}_i \neq {\beta}_j ~\mbox{and}~
   ~ ({\beta}_i + B {\mathbb{Z}}^2_N)\cap ({\beta}_j + B {\mathbb{Z}}^2_N)= \phi$.  
   Hence, $|\mathfrak{D}|=\Big|\displaystyle \frac{{\mathbb{Z}}^2_N}{B{\mathbb{Z}}^2_N}\Big| ~\mbox{and we can write},
   \displaystyle{\mathbb{Z}}^2_N=
  \bigcup_{{\beta} \,\in \,\mathfrak{D}}({{\beta} + B{\mathbb{Z}}^2_N}).$ The set $\mathfrak{D}$ satisfies some 
  nice properties, which are discussed in the following result:
 \begin{theorem}
 \label{thm 2.6} With the assumption of Theorem~\ref{thm}, let us consider the set  
   $ \mathfrak{D} \subseteq {\mathbb{Z}}^2_N$ defined as above with $(0,0)^t$ element having property that
   $ Ad_1\neq Ad_2 ~\mbox{whenever}~ d_1 \neq d_2 \in \mathfrak{D}.$  Then, 
  we can choose $\mathfrak{D}$ satisfying following properties:
  \begin{itemize}
   \item [(i)] $|\mathfrak{D}|=|A {\mathbb{Z}}^2_N|=|A(\mathfrak{D})|= \displaystyle \frac{N^2}{|\det(A)|}.$ \ \ \quad 
    (ii) $A {\mathbb{Z}}^2_N=A(\mathfrak{D}).$ 
   \item [(iii)] For some $\beta \in \mathfrak{D},
   A\beta=(0,0)^t ~\mbox{implies that}~\beta=(0,0)^t$.
   \item [(iv)]For $ \displaystyle \overline{\beta}_1, \overline{\beta}_2 \in
   \Big\{\overline {\beta}=\beta +  B{\mathbb{Z}}^2_N :\beta
    \in \mathfrak{D} \Big\}, ~
 \overline{\beta}_1 \bigcap \overline{\beta}_2= \phi,$ 
  and we can write $ \displaystyle{\mathbb{Z}}^2_N=
  \bigcup_{\beta \,\in \,\mathfrak{D}}(\beta + B{\mathbb{Z}}^2_N).$
  \item [(v)] For $ \beta_1, \beta_2 \in \mathfrak{D},~ \mbox{we have}~ 
  (\beta_1+\beta_2)+ B {\mathbb{Z}}^2_N= \beta +  B {\mathbb{Z}}^2_N$ 
  for some $ \beta \in \mathfrak{D}.$
  \item [(vi)] For $\beta_1,\beta_2 \in \mathfrak{D}, ~ \mbox{there exists}~
  \beta \in \mathfrak{D} 
  ~\mbox{such that}~ A\beta_1-A\beta_2=A\beta.$ Moreover,
  $\beta_1=\beta_2$ if and only if $\beta=(0,0)^t.$
  \end{itemize}
 \end{theorem}

\begin{proof}
	The explanation for the above properties is as follows:
	\begin{itemize}
		\item [(i)] From Theorem~\ref{thm} and above discussion about the set of digits,  we have 
		$|\mathfrak{D}|=\left|\frac{{\mathbb{Z}}^2_N}{B{\mathbb{Z}}^2_N}\right| =|\det (B)|
		= \frac{N^2}{|\det (A)|}=|A{\mathbb{Z}}^2_N|$, and hence, $|\mathfrak{D}|=|A(\mathfrak{D})|$ by noting the fact  that      
		$Ad_1\neq Ad_2$ whenever $ d_1\neq d_2 \in \mathfrak{D}$.
		
		\item [(ii)] The result $A {\mathbb{Z}}^2_N=A(\mathfrak{D})$ follows by noting that $A(\mathfrak{D})\subseteq A{\mathbb{Z}}^2_N$ and $|A(\mathfrak{D})|=|A {\mathbb{Z}}^2_N|$. 
		
		\item [(iii)] For this part, let us assume by contrary that there exists some $\beta_1\neq (0,0)^t \in \mathfrak{D}$ 
		such that
		$ A\beta_1=(0,0)^t.$ But it is given that
		$\beta_1\neq (0,0)^t \in \mathfrak{D}~\mbox{implies}~
		A\beta_1\neq (0,0)^t
		$, which is  a contradiction. 
		\item [(iv)] The result is obvious which follows from the definition of $\mathfrak{D}$.
		\item [(v)]
		For $\beta_1,\beta_2 \in \mathfrak{D}$,   we have 
		$(\beta_1 + 
		B{\mathbb{Z}}^2_N) + (\beta_2 + B{\mathbb{Z}}^2_N) = (\beta_1
		+\beta_2) + B{\mathbb{Z}}^2_N \in  \frac{{\mathbb{Z}}^2_N}{B{\mathbb{Z}}^2_N}$ and hence, we can write
		$(\beta_1+ \beta_2)+ B {\mathbb{Z}}^2_N= \beta +  B {\mathbb{Z}}^2_N$ 
		for some $ \beta \in \mathfrak{D}.$
		\item [(vi)] In view of the fact $A(\mathfrak{D})= A {\mathbb{Z}}^2_N,$ and for $\beta_1,\beta_2 \in \mathfrak{D}$, we have 
		$A(\beta_1)-A(\beta_2) \in A {\mathbb{Z}}^2_N,$ and hence there exists
		$\beta \in \mathfrak{D}~\mbox{such that}~ A(\beta_1)-A(\beta_2)=
		A\beta.$ For the remaining  part,   let $\beta_1=\beta_2$.
		Then $A(\beta_1)=
		A(\beta_2),$ which implies that 
		$A(\beta_1)-A(\beta_2) 
		=A(\beta)= (0,0)^t,$ and hence $\beta= (0,0)^t.$ Conversely, 
		suppose $\beta=(0,0)^t.$ This means, $A(\beta)=  A(\beta_1)-A(\beta_2)
		= (0,0)^t,~\mbox{which says that} ~\mbox{either}~ \beta_1=\beta_2 ~\mbox{or}~  
		A(\beta_1)= A(\beta_2), \mbox{but}~ A(\beta_1)= A(\beta_2)$ 
	 is possible only for  $\beta_1= \beta_2.$   
	\end{itemize} 
\end{proof}
  In order to understand about $\mathfrak{D}$, we provide following example:
   \begin{example}
  \label{eg 2.8}
   Let $N=2$ and $~A = \Big({\begin{matrix}
~2&2~ \\
~ 1&2 ~
\end{matrix}}\Big).$ Then,
$A {\mathbb{Z}}^2_2=
B {\mathbb{Z}}^2_2=\Big\{(0,0)^t,(0,1)^t\Big \},$ 
where
$B=NA^{-1}=
\Big({\begin{matrix}
~2&-2~ \\
~-1&2 ~
\end{matrix}}\Big)$, and hence 
$\displaystyle|A {\mathbb{Z}}^2_2|=|B{\mathbb{Z}}^2_2|=2 ~\mbox{and}~ 
   \Big|\frac{{\mathbb{Z}}^2_2}{B{\mathbb{Z}}^2_2}\Big|
   =|\det (B)|=2.$
Further, we have
$$
(0,0)^t +B{\mathbb{Z}}^2_2=\overline{(0,0)^t}=\overline{(0,1)^t}=
\Big\{(0,0)^t,(0,1)^t \Big \},
\overline{(1,0)^t}=\overline{(1,1)^t}=\Big\{ (1,0)^t,(1,1)^t \Big \},
$$
and hence, we can choose $ ~\mathfrak{D}=\Big\{(0,0)^t,(1,0)^t\Big \} \subset {\mathbb{Z}}^2_2,$
where $(0,0)^t \in \mathfrak{D}$ and $Ad_1\neq Ad_2$ whenever $ d_1\neq d_2 \in \mathfrak{D}$.
 Now, it can be easily checked that $\mathfrak{D}$ 
  satisfies all properties of Theorem~\ref{thm 2.6}. 
 \end{example}
 In the rest of  paper, we assume that {\it{the set of digits $\mathfrak{D}$ satisfies properties of Theorem~\ref{thm 2.6}}}. 
 Further, by the set $\mathfrak{D}^* \subseteq ({\mathbb{Z}}^2_N \setminus C{\mathbb{Z}}^2_N) \cup \{(0,0)^t\}$, we denote
 {\it{the set of digits}} corresponding to the cosets of $C {\mathbb{Z}}^2_N ~\mbox{in}~{\mathbb{Z}}^2_N,$ where 
 $C=B^t=:$ transpose of the matrix $B$. Note that $|\mathfrak{D}^*|=|\mathfrak{D}|$ and we can write  
 $\displaystyle{\mathbb{Z}}^2_N= \bigcup_{m \,\in \,\mathfrak{D}^*}( m + C{\mathbb{Z}}^2_N).$
 
 Now, we are in position to provide an  answer about the question for the system $\mathfrak B(\Phi)$ defined by (\ref{eq 1.1}). 
   Theorem~\ref{thm 2.6} 
implies that $|A(\mathfrak{D})|=|\mathfrak{D}|=\displaystyle \frac{N^2}{|\det(A)|}$ and from (\ref{eq 1.1}),
we have $|\mathfrak{B}(\Phi)|=|AI_1||I_2|.$ In order to be an orthonormal basis for 
 $\ell^2 ({\mathbb{Z}}^2_N)$, $\mathfrak{B}(\Phi)$ should have $N^2$ distinct elements forming orthonormal set, 
 hence $N^2=|\mathfrak{B}(\Phi)|=|AI_1||I_2|$. Further, by assuming $I_1=\mathfrak{D}$, we get 
$N^2=|A(\mathfrak{D})||I_2|$, 
which implies  $|I_2|=|\det(A)|$ in view of the fact that $|AI_1|=|A(\mathfrak{D})|=\displaystyle \frac{N^2}{|\det(A)|}$. Now, first   we consider a collection whose elements are 
translations of a single vector $\varphi_0 \in \ell^2(\mathbb Z_N^2)$.  Then,  we have $|I_2|=1$, that is, $|\det(A)|=1,$  which implies $|AI_1|=N^2$, 
and hence in this case, we must have $I_1= \mathfrak{D}={\mathbb{Z}}^2_N$. Therefore,  the system defined 
in (\ref{eq 1.1}) takes the following form:
\begin{equation}\label{eq 2.1}
\mathfrak{B}(\varphi_0) := \{ T_{Ak} \varphi_{0}: k \in {\mathbb{Z}}^2_N \} = \{ T_{k}\varphi_{0}: k \in {\mathbb{Z}}^2_N \}
\subset \ell^2 ({\mathbb{Z}}^2_N).
\end{equation}

 Next, we want to   characterize $\varphi_0$ such that the system defined in (\ref{eq 2.1}) yields an orthonormal basis for 
$\ell^2 ({\mathbb{Z}}^2_N)$ having time-frequency localization properties. For this, we need   following definition and result:

The \textit{discrete Fourier transform} (DFT) and \textit{Inverse discrete Fourier transform} (IDFT) on $\ell^2 ({\mathbb{Z}}^2_N)$ are 
defined for $m,n \in 
{\mathbb{Z}}^2_N,$  and $f,g \in \ell^2 ({\mathbb{Z}}^2_N)$ by
$$
\displaystyle \widehat{f}(m)=  \sum_{s
\,\in \, {\mathbb{Z}}^2_N}f(s)e^{-2\pi i \langle{m,s}\rangle/N}
~\mbox{and}~
{g}^{\vee}(n)=\frac{1}{N^2} \sum \limits_{l \in {\mathbb{Z}}^2_N}g(l)
e^{2\pi i \langle{n,l}\rangle/N},
$$
respectively. Further, we have $(\widehat{f})^{\vee}(n)=f(n)$,
for all $n \in {\mathbb{Z}}^2_N$ and the
Plancherel's formula 
 $$
 \langle{f,g}\rangle=\displaystyle \frac{1}{N^2} \sum_{m \,\in \,
 {\mathbb{Z}}^2_N}\widehat{f}(m)\overline{\widehat{g}(m)}
 =\displaystyle \frac{1}{N^2} \langle{\widehat{f},\widehat{g}}\rangle,
 $$
 which provides Parseval's formula for $f=g.$ 
 \\
   
 The following lemma   can be proved easily: 
\begin{lemma}{\label{le 2.10}} 
For $ A \in \mathcal{M}(2, \mathbb{Z})$,   $k,k_1\in {\mathbb{Z}}^2_N$ and $f,g \in \ell^2({\mathbb{Z}}^2_N)$, 
we have following properties: 
\begin{itemize}
\item[(i)]  $\widehat{(T_{Ak}f)}(m)=e^{-2\pi i \langle{m,Ak}\rangle/N}\widehat{f}(m)$ for all $m \in {\mathbb{Z}}^2_N$. 
\item[(ii)]  $ \langle{T_{Ak_1}f,T_{Ak}g}\rangle = \langle{f,T_{(Ak-Ak_1)}g}\rangle$. 
\item[(iii)]  $\widehat{\delta}(k)=1,~\mbox{where}~\delta(k)=  1$ for $k=(0,0)^t,$ and zero otherwise. 
  \end{itemize}
   \end{lemma}

In order to be an orthonormal basis for $\ell^2({\mathbb{Z}}^2_N)$, the system $\mathfrak{B}(\varphi_0)$ as defined in 
(\ref{eq 2.1}) has to satisfy 
following condition:
\begin{theorem}\label{thm 2.11}
The system $\mathfrak{B}(\varphi_0)$ forms an orthonormal basis for $\ell^2({\mathbb{Z}}^2_N)$ if and only if we have
  $|\widehat{\varphi_0}(k)|^2=1 ~\mbox{for all}~ k \in {\mathbb{Z}}^2_N.$ Moreover, orthonormal basis of this form is not 
  frequency localized.
 \end{theorem}
\begin{proof} The system  $\mathfrak{B}(\varphi_0)$ forms an orthonormal basis for $\ell^2({\mathbb{Z}}^2_N)$ if and only if 
 $\delta(k)=\langle{\varphi_0,T_{k}\varphi_0}\rangle $, for $k\in {\mathbb{Z}}^2_N$. By using the
 Plancherel's formula, we can write
\begin{equation*}
\delta(k)= \displaystyle \frac{1}{N^2}\langle{\widehat{\varphi_0},
\widehat{T_{k}\varphi_0}}\rangle
= \displaystyle \frac{1}{N^2}\sum_{n \in {\mathbb{Z}}^2_N}|\widehat{\varphi_0}(n)|^2
e^{2\pi i \langle{n,k}\rangle/N}= {\mathcal{G}}^\vee (k),
\end{equation*}
where  
$\mathcal{G}(n)=|\widehat{\varphi_0}(n)|^2$. Hence, 
the system  $\mathfrak{B}(\varphi_0)$ forms an orthonormal basis for $\ell^2({\mathbb{Z}}^2_N)$ if and only if $
\delta(k)={\mathcal{G}}^\vee(k)$, which is if and only if $|\widehat{\varphi_0}(k)|^2=1
~\mbox{for all}~ k \in {\mathbb{Z}}^2_N,$ in view of Lemma~\ref{le 2.10}. This implies that the vector 
$\varphi_0$ is not frequency localized, and hence the orthonormal basis of the form $\mathfrak{B}(\varphi_0)$ is 
not frequency localized.
 \end{proof}
\noindent

 In order to get 
  time as well as frequency localized orthonormal basis for $\ell^2({\mathbb{Z}}^2_N)$,
 we modify our approach by considering a collection whose elements are translations of 
 two or more vectors. 
 This means, we consider the case when $|I_2|=|\det(A)|>1$. From Theorem~\ref{thm 2.6}, 
 we have $|AI_1|=|A(\mathfrak{D})|=\frac{N^2}{|\det(A)|}$, 
 and hence in this case, we must have   $I_1=\mathfrak{D} \subsetneq \mathbb Z_N^2.$ Therefore, 
 for $|\det(A)|=q\geq 2$ and $\Phi=\{\varphi_p\}_{p=0}^{q-1}
 \subset \ell^2({\mathbb{Z}}^2_N),$ the system in (\ref{eq 1.1}) can be redefined as:
\begin{equation}\label{eq 2.2}
\mathfrak{B}(\Phi) := \{ T_{Ak} \varphi_{p}: k \in \mathfrak{D},\, 0\leq p \leq q-1 \}
\subset \ell^2 ({\mathbb{Z}}^2_N).
\end{equation}

\begin{remark}\label{re 2}
 An important point that one should keep in mind while choosing $\Phi$ in (\ref{eq 2.2}) is 
 to select $\Phi$ in such a way that there should not exist any element
$\psi \in \ell^2({\mathbb{Z}}^2_N)$ 
such that $\{ T_{k}\psi: k \in {\mathbb{Z}}^2_N \} = \mathfrak{B}(\Phi),$ otherwise $\mathfrak{B}(\Phi)$ will become similar
with the system (\ref{eq 2.1}), and  Theorem~\ref{thm 2.11} says that 
orthonormal basis of this form is not frequency localized.
For a better explanation of this fact, we have the 
following result which is also   useful in the sequel:
\end{remark}
\begin{proposition} \label{le 2.9}
 Let $\{ e_{m}\}_{m \in {\mathbb{Z}}^2_N}$ be a standard orthonormal basis for $\ell^2({\mathbb{Z}}^2_N)$, 
 where for each $j \in {\mathbb{Z}}^2_N,$ we define $e_j(n) = 1 ~\mbox{if}~ n = j$ and $0$ otherwise.
 Then, for $ N \in \mathbb{N}$ and $A \in {GL}(2, \mathbb{R})$ such that $A {\mathbb{Z}}^2 \subseteq{\mathbb{Z}}^2$ with
 $|\det(A)|=q\geq 2$, there exists $\{ \gamma_j\}_{j=0}^{q-1} \subset {\mathbb{Z}}^2_N$ for which 
 $\{ T_{Ak}e_{\gamma_j}: k \in \mathfrak{D},\, 0\leq j \leq q-1 \}=\{ T_k e_{n}\}_{k \in {\mathbb{Z}}^2_N}$, 
 for some $n \in {\mathbb{Z}}^2_N$. 
\end{proposition}
\begin{proof} Let $n \in {\mathbb{Z}}^2_N$. Then, both the collections $\{ T_k e_{n}\}_{k \in {\mathbb{Z}}^2_N}$ 
and $\{ e_{m}\}_{m \in {\mathbb{Z}}^2_N}$ are same, and hence we claim that 
$\{ T_{Ak}e_{\gamma_j}: k \in \mathfrak{D},\, 0\leq j \leq q-1 \}= \{ e_{m}\}_{m \in {\mathbb{Z}}^2_N}$. 
For this, it is enough to show that there exists 
$\{ \gamma_j\}_{j=0}^{q-1} \subset {\mathbb{Z}}^2_N$ such that the collection
$\{ {\gamma_j+ Ak}: k \in \mathfrak{D},\, 0\leq j \leq q-1 \}= {\mathbb{Z}}^2_N$, since 
 for $k \in \mathfrak{D}$ and $ 0\leq j \leq q-1,$ we have $ T_{Ak}e_{\gamma_j}= e_{(\gamma_j+ Ak)}$.
 Now, from Theorem \ref{thm}, we have 
  $\Big|\displaystyle \frac{{\mathbb{Z}}^2_N}{A{\mathbb{Z}}^2_N}\Big| =|\det (A)|=q$, and hence,  we have $q$ distinct coset representatives of   $ {A {\mathbb{Z}}^2_N}$ in ${\mathbb{Z}}^2_N$, say, $\{\overline{\alpha}_j\}_{j=0}^{q-1}$, 
  where for
  $0\leq j \leq q-1,~\mbox{we define}~\overline{\alpha}_j= {\alpha}_j + A{\mathbb{Z}}^2_N$. 
    Here, the set  
    $\{ \alpha_j\}_{j=0}^{q-1}=: 
    \mathfrak{D}_0  \subseteq ({\mathbb{Z}}^2_N\setminus A{\mathbb{Z}}^2_N) \cup \{(0,0)^t\}$   has  following properties:  
   ${\alpha}_{j_1} \neq {\alpha}_{j_2} ~\mbox{and}~
   ~ ({\alpha}_{j_1} + A{\mathbb{Z}}^2_N)\cap ({\alpha}_{j_2} + A {\mathbb{Z}}^2_N)= \phi$, 
   for $0\leq j_1 \neq j_2 \leq q-1$.  
   Hence, $|\mathfrak{D}_0|=\Big|\displaystyle \frac{{\mathbb{Z}}^2_N}{A{\mathbb{Z}}^2_N}\Big| ~\mbox{and we can write}, \,
   \displaystyle{\mathbb{Z}}^2_N=
  \bigcup_{\alpha \,\in \,\mathfrak{D}_0}({\alpha + A{\mathbb{Z}}^2_N})=
  \{ \alpha +  Ak: \alpha \in \mathfrak{D}_0,\, k \in \mathfrak{D} \}$. This implies that there exists 
  $\{ \alpha_j\}_{j=0}^{q-1}= 
   \mathfrak{D}_0 \subset {\mathbb{Z}}^2_N $ such that 
   $\{ T_{Ak}e_{\alpha_j}: k \in \mathfrak{D},\, 0\leq j \leq q-1 \}= \{ e_{m}\}_{m \in {\mathbb{Z}}^2_N}$. 
\end{proof}
Next, we are  defining  an orthonormal wavelet system for $\ell^2({\mathbb{Z}}^2_N)$:

\begin{definition}
\label{de 2.12} {\bf{Orthonormal Wavelet System:}} We call the system $\mathfrak{B}(\Phi)$, defined in 
(\ref{eq 2.2}), an {\it{orthonormal wavelet system (ONWS)}} for $\ell^2({\mathbb{Z}}^2_N)$, 
if it forms an orthonormal basis for 
$\ell^2({\mathbb{Z}}^2_N).$
\end{definition}
Following is our main  result of this section which characterizes $\Phi \subset \ell^2 (\mathbb Z_N^2)$ such that 
the system $\mathfrak{B}(\Phi)$ defined in (\ref{eq 2.2}) provides  an ONWS for $\ell^2({\mathbb{Z}}^2_N)$: 
\begin{theorem}\label{thm 2.15}
Let $~A \in {GL}(2, \mathbb{R})$ be a matrix such that $A {\mathbb{Z}}^2 \subseteq{\mathbb{Z}}^2 $
   and let   $C=(N {A}^{-1})^t \in {GL}(2, \mathbb{Z})$ for some $ N \in \mathbb{N}.$ Consider
   $\Phi=\{\varphi_{p}\}_{p=0}^{q-1} \subset \ell^2({\mathbb{Z}}^2_N)$, 
where $|\det(A)|= q\geq 2.$ Then, the following statements are equivalent:
   \begin{itemize}
    \item [(i)] $\mathfrak{B}(\Phi)\subset \ell^2({\mathbb{Z}}^2_N)$ forms an ONWS in 
    $\ell^2({\mathbb{Z}}^2_N)$. 
 \item[(ii)] 
 For all $k \in \mathfrak{D} ~\mbox{and}~ (p_1,p_2)
  \in \{(m,n):\, 0\leq m\leq n\leq q-1\},$
$$ \displaystyle \sum_{\gamma \in C {\mathbb{Z}}^2_N}\widehat{\varphi}_{p_1}(k+\gamma)
\overline{\widehat{\varphi}_{p_2}(k+ \gamma)}= q \delta_{p_1p_2}.$$
     \item[(iii)] The system matrix $\mathcal{S}_{\Phi}(k)$ of $\Phi$ is unitary for 
    each $k \in \mathfrak{D},$ where 
 \[
\displaystyle \mathcal{S}_{\Phi}(k)=\frac{1}{\sqrt{q}}
{{\begin{pmatrix}
\widehat{\varphi_{p}}(k +\gamma)
\end{pmatrix}}}_{\substack{
   \gamma \in C{\mathbb{Z}}^2_N \\
   0\leq p \leq q-1
  }}.
  \]
\end{itemize}
\end {theorem}
 For this, we need to set up   the background in the form of following results:
 \begin{lemma}
 \label{lemma 2.1} For $\alpha \in A {\mathbb{Z}}^2_N$ and 
 $\beta \in C{\mathbb{Z}}^2_N$, the inner product $\langle{\alpha, \beta}\rangle \in N \mathbb{Z}$, where $A~\mbox{and}~C$
 are as defined in Theorem~\ref{thm 2.15}.
   \end{lemma}
   \begin{proof}
  Let $\alpha=(\alpha_1,\alpha_2)^t \in A {\mathbb{Z}}^2_N~\mbox{and}~\beta=(\beta_1,\beta_2)^t
   \in C {\mathbb{Z}}^2_N.~\mbox{Then, there exist}~ n=
   (n_1,n_2)^t$ and $ m=(m_1,m_2)^t\in {\mathbb{Z}}^2_N ~\mbox{such that}~ \alpha=An ~\mbox{and}~ \beta= C m$, where
   $~A = \Big({\begin{matrix}
~a&b~ \\
~ c&d ~
\end{matrix}}\Big)\in {GL}(2, \mathbb{R})$ is a matrix with $A {\mathbb{Z}}^2 \subseteq{\mathbb{Z}}^2$ and
  $ C= \Bigg({\begin{matrix}
~\frac{Nd}{\det(A)}& \frac{-Nc}{\det(A)}~ \\
~  \frac{-Nb}{\det(A)}& \frac{Na}{\det(A)} ~
\end{matrix}}\Bigg) \in{GL}(2, \mathbb{Z}).$
  Now, the result follows by observing that
$
  \langle{\alpha,\beta}\rangle
  = \Bigg \langle{\Big({\begin{matrix}
~a&b~ \\
~ c&d ~
\end{matrix}}\Big)\Big({\begin{matrix}
~n_1~ \\
~ n_2 ~
\end{matrix}}\Big), \Bigg({\begin{matrix} 
~\frac{Nd}{\det(A)}&  \frac{-Nc}{\det(A)}~ \\
~  \frac{-Nb}{\det(A)}&  \frac{Na}{\det(A)} ~
\end{matrix}}\Bigg)\Big({\begin{matrix}
~m_1~ \\
~ m_2 ~
\end{matrix}}\Big)}\Bigg \rangle 
= N p,  ~\mbox{for}~ p = n_1m_1+n_2m_2 \in \mathbb{Z}.  
 $ \end{proof}
\begin{proposition}
\label{thm 2.5} Let
 $\{{E}_{k}\}_{k \in \mathfrak{D}} \subseteq \ell^2(\mathfrak{D}^*)$,
 where $\mbox{for each}~ k \in \mathfrak{D}$ and $A$ as defined in Theorem~\ref{thm 2.15}, we define
  $E_k(m)= \frac{1}{\sqrt{|\mathfrak{D}^*|}} 
  e^{2 \pi i {\langle{m, A k}\rangle}/N}$ for all $ m \in \mathfrak{D}^*$. Then, the system $\{{E}_{k}\}_{k \in \mathfrak{D}}$ 
 forms an orthonormal basis 
 for $\ell^2(\mathfrak{D}^*)$.
\end{proposition}

\begin{proof}
Observe that for $k_1,k_2 \in \mathfrak{D} $, we have
 \begin{equation}\label{eq 2.4}
 \langle{E_{k_1},E_{k_2}} \rangle =
  \displaystyle \sum_{m \in \mathfrak{D}^*} E_{k_1}(m) 
  \overline{E_{k_2}(m)}
 = \frac{1}{|\mathfrak{D}^*|} \sum_{m \in \mathfrak{D}^*}
 e^{2 \pi i \langle{m, (Ak_1-Ak_2)}\rangle/N}.
   \end{equation}
 We claim that
  \begin{equation}\label{eq 2.0}
 \sum_{m \in \mathfrak{D}^*} e^{2 \pi i \langle{m, (Ak_1-Ak_2)}\rangle/N}
   = \left\{
   \begin{array}{lr}
    |\mathfrak{D}^*|  \quad  \mbox{for} ~k_1=k_2,\\   
    0  \quad  \quad \ \mbox{otherwise} .
   \end{array}
   \right.
       \end{equation}
   For this, let $m_1 \in \mathfrak{D}^*$ be an arbitrary element. Then, we can write
   \begin{align*}
      e^{2 \pi i \langle{m_1, (Ak_1-Ak_2)}\rangle/N}\sum_{m \in \mathfrak{D}^*}
      e^{2 \pi i \langle{m, (Ak_1-Ak_2)}\rangle/N}
      =  \sum_{m \in \mathfrak{D}^*} e^{2 \pi i \langle{(m+m_1), 
      (Ak_1-Ak_2)
      }\rangle/N}. 
       \end{align*}
  Now, in view of Theorem~\ref{thm 2.6}, we have $Ak_1-Ak_2=Ak$, for some $k \in \mathfrak{D}$. Therefore,   by substituting $m+m_1=m_2$, we get  
      \begin{align*}
      e^{2 \pi i \langle{m_1, Ak}\rangle/N}
      \sum_{m \in \mathfrak{D}^*} 
      e^{2 \pi i \langle{m, Ak}\rangle/N}
      &=\sum_{m \in \mathfrak{D}^*} e^{2 \pi i \langle{m + m_1, Ak}\rangle/N}
       =\sum_{m_2 \in \mathfrak{D}^*} e^{2 \pi i \langle{m_2, Ak}\rangle/N},
    \end{align*} 
     which implies that either $e^{2 \pi i \langle{m_1, Ak}\rangle/N}=1$, or,
       $ \sum \limits_{m \in \mathfrak{D}^*} e^{2 \pi i \langle{m, Ak}\rangle/N}=0.$ But,
       $ e^{2 \pi i \langle{m_1, Ak}\rangle/N}=1$ if and only if
       $\langle{m_1,Ak}\rangle \in N\mathbb{Z},$
        which is if and only if we have
       $Ak \in N{\mathbb{Z}}^2 
       ~\mbox{as}~ m_1 \in \mathfrak{D}^* ~\mbox{is arbitrary}$. Equivalently, we can say that $e^{2 \pi i \langle{m_1, Ak}\rangle/N}=1$ 
        if and only if $k \in NA^{-1}{\mathbb{Z}}^2 \cap \mathfrak{D}= \{(0,0)^t\},$ which in view of 
       Theorem~\ref{thm 2.6} implies that $k_1=k_2.$  
Now, by using (\ref{eq 2.0}) and  (\ref{eq 2.4}), we get $\langle{E_{k_1},E_{k_2}} \rangle = \delta_{k_1 k_2}$. 
Hence the result follows.
\end{proof}
\begin{proof}[Proof of Theorem~\ref{thm 2.15}]
The system $\mathfrak{B}(\Phi)$ will form an orthonormal basis for $\ell^2({\mathbb{Z}}^2_N)$ if and only if for
$k \in \mathfrak{D}~\mbox{and} ~\varphi_{p_1},\varphi_{p_2} \in \Phi,~\mbox{where}~ 0 \leq p_1,p_2\leq q-1$, we have
\begin{align*}
\delta_{p_1p_2}\delta(k)=\langle{\varphi_{p_1},T_{Ak} \varphi_{p_2}}\rangle 
&=\displaystyle \frac{1}{N^2}\langle{\widehat{\varphi_{p_1}}, 
\widehat{T_{Ak} \varphi_{p_2}}}\rangle 
= \displaystyle \frac{1}{N^2}\sum_{n \in {\mathbb{Z}}^2_N}\widehat{\varphi_{p_1}}(n)\overline{\widehat{\varphi_{p_2}}(n)}
e^{2\pi i \langle{n, Ak}\rangle/N},
\end{align*}
in view of the Plancherel's formula. Next, by applying  
Lemma~\ref{lemma 2.1} in the above equation, we get 
\begin{align*}
\delta_{p_1p_2}\delta(k)&=\displaystyle \frac{1}{N^2}\sum_{m \in \mathfrak{D}^*} \sum_{\gamma \in C{\mathbb{Z}}^2_N}
\widehat{\varphi_{p_1}}(m+\gamma)\overline{\widehat{\varphi_{p_2}}(m+ \gamma)}
e^{2\pi i \langle{(m+ \gamma), Ak}\rangle/N}\\
&=\displaystyle \frac{1}{N^2}\sum_{m \in \mathfrak{D}^*}
\sum_{\gamma \in C{\mathbb{Z}}^2_N}\widehat{\varphi_{p_1}}(m+\gamma)\overline{\widehat{\varphi_{p_2}}(m+ \gamma)} 
e^{2\pi i \langle{m, Ak}\rangle/N}\\
&= \displaystyle \frac{1}{q|\mathfrak{D}^*|}\sum_{m \in \mathfrak{D}^*}
\Psi(m)e^{2\pi i\langle{m, Ak}\rangle/N},
\end{align*}
where 
$\displaystyle \Psi(m)=\sum_{\gamma \in C{\mathbb{Z}}^2_N}
\widehat{\varphi_{p_1}}(m+\gamma)\overline{\widehat{\varphi_{p_2}}(m+ \gamma)},$ for all $m \in \mathfrak{D}^*,$
and hence, for all $
k \in \mathfrak{D}, 
$ we have  $ 
 \delta_{p_1p_2}\delta(k)= \displaystyle \frac{1}{q}{\mathcal{F}}^{-1}(\Psi(k)),$   where  
${\mathcal{F}}^{-1}$ denotes the Inverse discrete Fourier transform on $\ell^2(\mathfrak{D}^*)$.  Next,  
 by noting   $ 
 (\Psi(m))_{m \in \mathfrak{D}^*} \in \ell^2(\mathfrak{D}^*)$,  
    Proposition~\ref{thm 2.5} and Lemma~\ref{le 2.10}, we get $\Psi(k)=q\delta_{p_1p_2}$.  
Hence, the system $\mathfrak{B}(\Phi)$ forms an ONWS in $\ell^2({\mathbb{Z}}^2_N)$ if and only if for
$k \in \mathfrak{D}~\mbox{and}~\varphi_{p_1},\varphi_{p_2} \in \Phi,$ where $ 0 \leq p_1,p_2\leq q-1$, we have
\begin{equation*}
 \mathfrak{R}_{p_1p_2}:=\displaystyle \sum_{\gamma \in C{\mathbb{Z}}^2_N}
\widehat{\varphi_{p_1}}(k+\gamma)\overline{\widehat{\varphi_{p_2}}(k+ \gamma)}=q\delta_{p_1p_2} ~\mbox{for all}~k \in \mathfrak{D}.
\end{equation*}
Further, we note that $\mathfrak{R}_{p_2p_1}=\overline{\mathfrak{R}}_{p_1p_2}$, which proves (i) $\Leftrightarrow$ (ii) part. 
Observe that above equation is equivalent to the fact that for $k \in \mathfrak{D}$, columns of $\mathcal{S}_{\Phi}(k)$, 
the system matrix of $\Phi$ having 
order $ q \times q$, defined by 
\[
\displaystyle \mathcal{S}_{\Phi}(k)=\frac{1}{\sqrt{q}}
{{\begin{pmatrix}
\widehat{\varphi_{p}}(k +\gamma)
\end{pmatrix}}}_{\substack{
   \gamma \in C{\mathbb{Z}}^2_N \\
   0\leq p \leq q-1
  }} ,
  \]
  forms an orthonormal basis for  ${\mathbb{C}}^q$. Equivalently, the collection 
  $\mathfrak{B}(\Phi)\subset \ell^2({\mathbb{Z}}^2_N)$ is an ONWS in
$\ell^2({\mathbb{Z}}^2_N)$ if and only if  
the system matrix of $\Phi$ is unitary for each $k \in \mathfrak{D}.$ Hence (i) $\Leftrightarrow$ (iii) 
 follows.
\end{proof}
 Following are some examples of ONWS in 
$\ell^2({\mathbb{Z}}^2_N)$ with respect to expansive as well as non-expansive matrices. 
In the case of $L^2(\mathbb R^n)$, the expansive and non-expansive 
nature of a dilation matrix plays an important role in the existence of an orthonormal wavelet. 
 A full characterization of dilation matrices which yields wavelets is still an open problem.  
\begin{example}\label{eg 2.12}
 (i) \textbf{For non-expansive matrix:} Let us recall Example~\ref{eg 2.8} in which $\mathfrak D=\{(0, 0)^t, (1, 0)^t\}$, and 
 the matrix $A$ is non-expansive since one of its  eigenvalues ($2 \pm \sqrt{2}$) is less than 1.  Consider $\Phi_1=\{\varphi_0,\varphi_1\} \subset \ell^2({\mathbb{Z}}^2_2)$   defined by  
$$\varphi_0={{\begin{pmatrix}
\varphi_0((0,0)^t) & \varphi_0((0,1)^t) \\
\varphi_0((1,0)^t) & \varphi_0((1,1)^t)
\end{pmatrix}}}={{\begin{pmatrix}
\frac{1}{\sqrt{2}} & 0 \\
\frac{1}{\sqrt{2}} & 0
\end{pmatrix}}},   
 ~\mbox{and}~
\varphi_1={{\begin{pmatrix}
\frac{1}{\sqrt{2}} & 0 \\
-\frac{1}{\sqrt{2}} & 0
\end{pmatrix}}}.$$
Then, the  collection 
$
\mathfrak{B}(\Phi_1)  = \{ T_{A k} \varphi_{p}: k \in \mathfrak{D},\, 0\leq p \leq 1\}
= \{\varphi_0, \varphi_1, T_{(0,1)^t}\varphi_0, T_{(0,1)^t}\varphi_1\}
$
is an ONWS for $\ell^2({\mathbb{Z}}^2_2)$, where 
$T_{(0,1)^t}\varphi_0={{\begin{pmatrix}
0 & \frac{1}{\sqrt{2}}\\
0 &\frac{1}{\sqrt{2}}
\end{pmatrix}}}
  ~\mbox{and}~
T_{(0,1)^t}\varphi_1={{\begin{pmatrix}
0 & \frac{1}{\sqrt{2}}\\
0 &-\frac{1}{\sqrt{2}}
\end{pmatrix}}}.$
 Further, the discrete Fourier transforms  of $\varphi_0$ and $\varphi_1$ are given by
 $\widehat{\varphi}_0={{\begin{pmatrix}
\sqrt{2} &\sqrt{2} \\
 0 & 0
\end{pmatrix}}}
~\mbox{and}~
\widehat{\varphi}_1={{\begin{pmatrix}
0&0  \\
\sqrt{2}& \sqrt{2}
\end{pmatrix}}},$ respectively.
 We can easily verify that for each $k \in \mathfrak{D}$ and $0\leq p_1,p_2\leq 1$,
  \begin{align*}
\displaystyle \sum_{\gamma \in C{\mathbb{Z}}^2_2}\widehat{\varphi}_{p_1}(k+ \gamma)
\overline{\widehat{\varphi}_{p_2}(k+ \gamma)}= 2 \delta_{p_1p_2},  
\end{align*} 
 where $C{\mathbb{Z}}^2_2=B^t{\mathbb{Z}}^2_2=\{(0,0)^t, (1,0)^t\}$. Hence, $\mathfrak{B}(\Phi_1)=
 \{ T_{Ak} \varphi_{p}: k \in \mathfrak{D},\, 0\leq p \leq 1 \}
\subset \ell^2({\mathbb{Z}}^2_2)$ forms an ONWS in $\ell^2({\mathbb{Z}}^2_2),$ by using Theorem~\ref{thm 2.15}.\\
(ii) \textbf{For expansive matrix:}
 Let $~A = \Big({\begin{matrix}~3&-1~ \\~ 1&1 ~\end{matrix}}\Big)$ be an expansive matrix with eigenvalues $2,2$.
 Then, $A {\mathbb{Z}}^2_4=B {\mathbb{Z}}^2_4=\Big\{  (0,0)^t,(1,3)^t,(2,2)^t,(3,1)^t \Big \},$
 where $B=NA^{-1}=\Big({\begin{matrix}~1&1~ \\~-1&3 ~\end{matrix}}\Big)$. 
 Further, we can choose the set of digits 
 $ ~\mathfrak{D}=\Big\{(0,0)^t,(0,1)^t,(0,2)^t,(0,3)^t\Big \} \subset {\mathbb{Z}}^2_4,$  
 which satisfies all properties from Theorem~\ref{thm 2.6}. Next, we consider $\Phi_2=\{ \varphi_0, \varphi_1, \varphi_2, \varphi_3 \} \subset \ell^2({\mathbb{Z}}^2_4)$
 whose   discrete Fourier transforms  are given by
{\footnotesize {$$
\widehat{\varphi_0}={{\begin{pmatrix}
\sqrt{2} & 0 & \sqrt{2}i & 0\\
-\sqrt{2}i & 0 & -\sqrt{2} & 0\\
0 & 1-i & 0 & -\sqrt{2}\\
0 & \sqrt{2} & 0 & -1-i
\end{pmatrix}}},~
\widehat{\varphi_1}={{\begin{pmatrix}
0 & \sqrt{2} & 0 & \sqrt{2}\\
0 & \sqrt{2}i & 0 & -\sqrt{2}i\\
-\sqrt{2} & 0 & \sqrt{2}i & 0\\
\sqrt{2}i& 0 & \sqrt{2} & 0
\end{pmatrix}}},
$$}}
{\footnotesize{
$$
\widehat{\varphi_2}={{\begin{pmatrix}
0 & \sqrt{2}i & 0 & \sqrt{2}i\\
0 & -1+i & 0 & 1-i\\
-\sqrt{2}i & 0 & 1-i & 0\\
\sqrt{2} & 0 & -\sqrt{2}i & 0
\end{pmatrix}}},
~\mbox{and}~ \
\widehat{\varphi_3}={{\begin{pmatrix}
-1-i & 0 & 1-i & 0\\
1+i & 0 & 1-i & 0\\
0 & -\sqrt{2}i & 0 & -1+i\\
0 & 1+i & 0 & -\sqrt{2}i
\end{pmatrix}}}.
  $$}}
   Then, for $0\leq p_1,p_2\leq 3,$ we can check that 
 $
\displaystyle \sum_{\gamma \in C{\mathbb{Z}}^2_4}\widehat{\varphi}_{p_1}(k+\gamma)
\overline{\widehat{\varphi}_{p_2}(k+ \gamma)}= 4 \delta_{p_1p_2}  
~\mbox{for all}~  k \in \mathfrak{D},
$ where $C{\mathbb{Z}}^2_4=B^t{\mathbb{Z}}^2_4=\{(0,0)^t,(1,1)^t,(2,2)^t,(3,3)^t \}$. Hence, from Theorem~\ref{thm 2.15}, we conclude that the collection 
$\mathfrak{B}(\Phi_2)=\{ T_{Ak} \varphi_{p}: k \in \mathfrak{D},\, 0\leq p \leq 3 \}
\subset \ell^2({\mathbb{Z}}^2_4)
$ forms an ONWS in $\ell^2({\mathbb{Z}}^2_4).$
 \end{example}
  In the next section, we discuss some results on uncertainty 
  principle by coupling Fourier basis and standard orthonormal basis 
  with the orthonormal wavelet system $\mathfrak{B}(\Phi)$ defined in (\ref{eq 2.2}). For this, let us denote the 
  standard orthonormal basis for $\ell^2({\mathbb{Z}}^2_N)$ by
  $\mathfrak{B}_S:=\{ T_{Ak}e_{\alpha_j}: k \in \mathfrak{D},\, 0\leq j \leq q-1 \}$, 
  where $\{ \alpha_j\}_{j=0}^{q-1}= \mathfrak{D}_0 \subset {\mathbb{Z}}^2_N $ (see Proposition~\ref{le 2.9}  for more details). By $\mathfrak{F}:=\{F_{v}\}_{ v \in {\mathbb{Z}}^2_N} $, we denote a two dimensional 
  Fourier basis for $\ell^2({\mathbb{Z}}^2_N)$ defined by $F_v(n)=\frac{1}{N}e^{2\pi i \langle{n,v}\rangle/N}$,  
  for all $n \in {\mathbb{Z}}^2_N$.
  For a non-zero vector $  f \in \ell^2({\mathbb{Z}}^2_N)$, we can  write 
    \begin{equation}\label{eq 2.5}
    f=\displaystyle \sum_{k \in \mathfrak{D}}\sum_{j=0}^{q-1} t_{j,k} T_{Ak}e_{\alpha_j} =   
    \displaystyle  \sum_{m \in \mathfrak{D}}\sum_{p=0}^{q-1} s_{p,m} T_{Am} \varphi_p=
    \displaystyle \sum_{v \in {\mathbb{Z}}^2_N}w_{v}F_{v},
    \end{equation}
    and hence, we have 
     \begin{equation}\label{eq 2.6}
     \|f\|^2=\displaystyle \sum_{k \in \mathfrak{D}}\sum_{j=0}^{q-1} |t_{j,k}|^2=   
     \displaystyle  \sum_{m \in \mathfrak{D}}\sum_{p=0}^{q-1} |s_{p,m}|^2 =
     \displaystyle \sum_{v \in {\mathbb{Z}}^2_N} |w_{v}|^2.
     \end{equation}
   Further, we denote the number of non-zero coefficients
    from among $\{t_{j,k}: k \in \mathfrak{D}, 0\leq j \leq q-1 \}$, $\{w_{v}: v \in \mathbb{Z}_N^2\}$
    and $\{s_{p,m}: m \in \mathfrak{D}, 0\leq p \leq q-1 \}$ by 
    $S_f,C_f$ and $W_f$, respectively. Note that \textit{max $X$} represents the maximum of all elements of the set
    $X\subset \mathbb R$.
\section{Uncertainty Principle corresponding to ONWS}\label{sec3}
Uncertainty principles put restrictions on how well frequency localized  a good time localized signal can be and vice versa. In the case of a signal defined on a finite abelian group, localization is generally
expressed through the cardinality of the support of the signal.  
Uniqueness of sparse representation of a signal depends upon the bound provided by uncertainty 
relations in terms of pair of bases.  
For the setup of $\ell^2({\mathbb{Z}}^2_N)$,   we prove following   results on the uncertainty principle with respect to   $\mathfrak{B}(\Phi)$  that can be generalized for the case of $\ell^2({\mathbb{Z}}^M_N), M \in \mathbb N$:
\begin{theorem} \label{thm 3.1}   Let  $\Phi=\{\varphi_p\}_{p=0}^{q-1} \subset \ell^2({\mathbb{Z}}^2_N)$ be such that 
$\mathfrak{B}(\Phi)$ is an ONWS in $\ell^2({\mathbb{Z}}^2_N)$.  Consider   two positive real numbers  $R_0$ and $E_0$    
defined by 
	\begin{align*} R_0&= \max\left\{|\varphi_p(\alpha +A\beta)|: \alpha \in \mathfrak D_0, \beta \in \mathfrak D,  0\leq   p \leq q-1 \right\}, 
	 \ and \\   E_0&=\max\displaystyle \Big\{\frac{1}{N}\sum_{n \in {\mathbb{Z}}^2_N}|\varphi_p(n)|: 0\leq p \leq q-1\Big\}.
	\end{align*} 
Then,  the following inequalities hold true:
		 \begin{itemize}
		 	\item [(i)] The  bounds for $R_0$ and $E_0$ are given by $\frac{1}{N} \leq R_0, E_0 \leq 1$. 
		 	In case of $R_0=1$, the system $\mathfrak{B}(\Phi)$ is not frequency localized.
		 \item [(ii)]	Representations of $f\in \ell^2({\mathbb{Z}}^2_N)$ in terms  of $\mathfrak{B}(\Phi)$ and
		 $\mathfrak{B}_S$  provide   following  relations:
		 $$
		 S_f W_f \geq \max\left\{2, \frac{1}{R_0^2}\right\},  \ \  \mbox{and} \ \  S_f + W_f \geq \max\left\{3, \frac{2}{R_0}  \right\}.$$
		 \item [(iii)]	Representations of $f\in \ell^2({\mathbb{Z}}^2_N)$ in terms  of $\mathfrak{B}(\Phi)$ and 
		 $\mathfrak{F}$  provide   following  relations:
		 $$
		 C_f W_f \geq \frac{1}{E_0^2},  \ \  \mbox{and} \ \  C_f + W_f \geq \frac{2}{E_0}.$$
		 	\end{itemize} 
	\end{theorem}
\begin{proof}
 By considering representations of $f\in \ell^2({\mathbb{Z}}^2_N)$ in terms of $\mathfrak{B}(\Phi)$ 
   and $\mathfrak{B}_S$ from (\ref{eq 2.5}), we have the following:
 \begin{align*}
 {||f||}^2  
 &= \left|\left\langle {\sum_{m \in \mathfrak{D}} \displaystyle \sum_{p=0}^{q-1} s_{p,m} T_{Am} \varphi_p, 
 \displaystyle \sum_{k \in \mathfrak{D}} \sum_{j=0}^{q-1} t_{j,k} T_{Ak}e_{\alpha_j}}\right\rangle\right| 
   = \left|\sum_{m \in \mathfrak{D}} \displaystyle \sum_{p=0}^{q-1} s_{p,m}
 \displaystyle \sum_{k \in \mathfrak{D}} \sum_{j=0}^{q-1} \overline{t_{j,k}}\left\langle{T_{Am} \varphi_p, T_{Ak}e_{\alpha_j}}
 \right\rangle\right|\\
 & \leq 
 \sum_{m \in \mathfrak{D}} \displaystyle \sum_{p=0}^{q-1}
 \displaystyle \sum_{k \in \mathfrak{D}} \sum_{j=0}^{q-1}
  |s_{p,m}| |\overline{t_{j,k}}||\langle{T_{Am} \varphi_p, T_{Ak}e_{\alpha_j}}
 \rangle|,
 \end{align*}
and  hence we have
 \begin{equation}\label{eq 3.2}
 {||f||}^2\leq \displaystyle \sum_{m \in \mathfrak{D}} \displaystyle \sum_{p=0}^{q-1} |s_{p,m}|
 \displaystyle \sum_{k \in \mathfrak{D}} \sum_{j=0}^{q-1}
  |\overline{t_{j,k}}||\langle{T_{Am} \varphi_p, T_{Ak}e_{\alpha_j}}
 \rangle|. 
 \end{equation}
 Similarly, if we proceed by using representations of $f\in \ell^2({\mathbb{Z}}^2_N)$ in terms of $\mathfrak{B}(\Phi)$ 
 and  $\mathfrak{F}$, we get 
 \begin{equation}\label{eq 3.3}
 {||f||}^2\leq \displaystyle \sum_{m \in \mathfrak{D}} \displaystyle \sum_{p=0}^{q-1} |s_{p,m}|
 \displaystyle \sum_{v \in \mathbb{Z}_N^2}
  |\overline{w_{v}}||\langle{T_{Am} \varphi_p, F_{v}}
 \rangle|. 
 \end{equation}
In view of Theorem \ref{thm 2.6}, we observe  that for $k,m \in \mathfrak{D}$ and $0\leq j,p \leq q-1 $,
 \begin{align*}
 |\langle{T_{Am} \varphi_p, T_{Ak}e_{\alpha_j}}\rangle|
 &=  |\langle{ \varphi_p, T_{(Ak-Am)}e_{\alpha_j}}\rangle| 
   =  |\langle{ \varphi_p, T_{A\beta}e_{\alpha_j}}\rangle|~\mbox{for some}~ \beta \in \mathfrak{D}\\
  &= |\mbox{trace}(({T_{A\beta}e_{\alpha_j}})^{*} \varphi_p)| 
   = \left|  \sum_{n \in {\mathbb{Z}}^2_N}\varphi_p(n)\overline{T_{A \beta} e_{\alpha_j}(n)} \right|\\
   &= \left|  \sum_{n \in {\mathbb{Z}}^2_N}\varphi_p(n)\overline{ e_{\alpha_j + A\beta}(n)} \right|
     = \left|\varphi_p(\alpha_j +A\beta)\right|,
 \end{align*}
 which implies that  $R_1=R_2$, for $R_1=\{|\langle{T_{Am} \varphi_p, T_{Ak}e_{\alpha_j}}\rangle|:k,m \in \mathfrak{D}, 0\leq j,p \leq q-1\}$,
 and $R_2=\left\{|\varphi_p(\alpha_j +A\beta)|: \alpha_j \in \mathfrak D_0, \beta \in \mathfrak D,  0\leq   p, j \leq q-1 \right\}$, and hence for any $h \in R_1$,
we have $h \leq R_0$, where $R_0=\max  R_2$. Using this fact in (\ref{eq 3.2}), along with (\ref{eq 2.6}) and  the Cauchy-Schwarz inequality, 
it is clear that 
 \begin{align*}
  {||f||}^2 & \leq \displaystyle \sum_{m \in \mathfrak{D}}  \sum_{p=0}^{q-1} |s_{p,m}|
 \displaystyle \sum_{k \in \mathfrak{D}} \sum_{j=0}^{q-1} |\overline{t_{j,k}}|R_0 
   \leq \sqrt{\displaystyle \sum_{m,p: |s_{p,m}|\neq 0} |s_{p,m}|^2 } \sqrt{\displaystyle \sum_{m,p: |s_{p,m}|\neq 0} 
\Big( \sum_{k \in \mathfrak{D}}\sum_{j=0}^{q-1} |\overline{t_{j,k}}|R_0\Big)^2}\\
 & = \sqrt{\displaystyle \sum_{m \in \mathfrak{D}}  \sum_{p=0}^{q-1} |s_{p,m}|^2 }
 \sqrt{ W_f \Big( \sum_{k \in \mathfrak{D}}\sum_{j=0}^{q-1} |\overline{t_{j,k}}|R_0\Big)^2} 
   = ||f||\sqrt{W_f}  \sum_{k \in \mathfrak{D}}\sum_{j=0}^{q-1} |\overline{t_{j,k}}|R_0\\
 & \leq ||f|| \sqrt{W_f} \sqrt{\displaystyle \sum_{j,k: |{t_{j,k}}|\neq 0} |\overline{t_{j,k}}|^2 } 
 \sqrt{\displaystyle \sum_{j,k: |t_{j,k}|\neq 0}{\big(R_0 \big)}^2} 
   = ||f||\sqrt{W_f} \sqrt{\displaystyle \sum_{k \in \mathfrak{D}}\sum_{j=0}^{q-1} |t_{j,k}|^2 } 
 \sqrt{S_f{\big(R_0 \big)}^2}\\
 &= {||f||}^2 \sqrt{W_f} \sqrt{S_f}R_0.
 \end{align*}
 Therefore, we have the inequality $\sqrt{S_f W_f}\geq \frac{1}{R_0}$. Further, by assuming part (i) [which we will prove later], we cannot consider $R_0$   equal to 1, otherwise the system $\mathfrak{B}(\Phi)$ will not remain time-frequency localized. Next, by using inequality of arithmetic and geometric means, we have 
 $S_f + W_f \geq 2 \sqrt{S_f W_f}\geq \frac{2}{R_0},$ which leads to   (ii).

 Next, we prove  (iii) part.  For this, first  we observe the following:
 \begin{align*}
 |\langle{T_{Am} \varphi_p, F_{v}}\rangle|
 &=  |\sum_{n \in {\mathbb{Z}}^2_N}T_{Am} \varphi_p(n)\overline{F_{v}(n)}| 
   =  |\sum_{n \in {\mathbb{Z}}^2_N}\varphi_p(n-Am)\frac{1}{N}e^{-2\pi i \langle{n,v}\rangle/N}| 
   \leq \frac{1}{N}\sum_{t \in {\mathbb{Z}}^2_N}|\varphi_p(t)|, 
 \end{align*}
for $m \in \mathfrak{D}, 0\leq p \leq q-1$ and $ v \in \mathbb{Z}_N^2$, which implies that   $\max X \leq \max Y$, 
 where  sets $X$ and $Y$ are given by $\{|\langle{T_{Am} \varphi_p, F_{v}}\rangle|:m \in \mathfrak{D}, 0 \leq p \leq q-1, v \in 
 {\mathbb{Z}}^2_N \}$ and 
 $\Big\{\displaystyle \frac{1}{N}\sum_{t \in {\mathbb{Z}}^2_N}|\varphi_p(t)|: 0\leq p \leq q-1\Big\},$ 
 respectively, and hence for any $x \in X$,
we have $x \leq E_0$, where $E_0=\max Y$. Therefore, we obtain the inequality $\sqrt{C_f W_f}\geq \frac{1}{E_0}$,    using the above discussion in (\ref{eq 3.3}) along with the way we proceeded in case of (ii).
 Now, by inequality of arithmetic and geometric means, we get
$C_f + W_f \geq 2 \sqrt{C_f W_f}\geq \frac{2}{E_0},$ and hence (iii) follows. 

For the   part (i), let us  consider  elements of $\mathfrak{B}(\Phi)$  and $\mathfrak{B}_S$. Then,   from Cauchy-Schwarz inequality, we have 
$$|\langle{T_{Am} \varphi_p, T_{Ak}e_{\alpha_j}}\rangle| \leq ||T_{Am} \varphi_p || ||T_{Ak}e_{\alpha_j} ||=1,$$  for all  $k,m \in \mathfrak{D}$ and $0\leq j,p \leq q-1.$ 
 Therefore, $1$ is an upper bound for the set $R_1$ and hence its maximum element $R_0\leq 1$. Next, we note that the real number  $R_0$ cannot take the value 1.  
For this, let by contradiction we assume $R_0=1$. Then, we have  
\begin{align*}1&=\max\left\{|\varphi_p(\alpha +A\beta)|: \alpha \in \mathfrak D_0, \beta \in \mathfrak D,  0\leq   p \leq q-1 \right\}\\
&= \max\{|\varphi_p(n)|: n \in {\mathbb{Z}}^2_N, 0\leq p \leq q-1 \},
\end{align*}
since the collection  $\{ \alpha +  Ak: \alpha \in \mathfrak{D}_0,\, k \in \mathfrak{D} \}$  is a partition of $\displaystyle{\mathbb{Z}}^2_N$. This  assures the existence of  
$p_1 \in \{0,1,\ldots, q-1\}$ and $n_1 \in \mathbb{Z}_N^2$ such that $|\varphi_{p_1}(n_1)|=1.$ Therefore, we have
\begin{equation} \label{eq 3.5}
 ||\varphi_{p_1}||^2=\displaystyle \sum_{m \in \mathbb{Z}_N^2}|\varphi_{p_1}(m)|^2= 
 \displaystyle |\varphi_{p_1}(n_1)|^2 + \sum_{m\neq n_1 \in \mathbb{Z}_N^2}|\varphi_{p_1}(m)|^2 =1,
\end{equation}
in view of the fact that $\varphi_{p_1} \in \Phi$ and $\mathfrak{B}(\Phi)$ forms an ONWS in $\ell^2(\mathbb{Z}_N^2).$ 
Since $|\varphi_{p_1}(n_1)|=1,$ therefore from 
(\ref{eq 3.5}), it is clear that $\displaystyle \sum_{m\neq n_1 \in \mathbb{Z}_N^2}|\varphi_{p_1}(m)|^2=0,$ which provides
$\varphi_{p_1}(m)=0$ for all $m\neq n_1 \in \mathbb{Z}_N^2,$ and hence we can define $\varphi_{p_1} \in \Phi \subset 
\ell^2(\mathbb{Z}_N^2)$ by $\varphi_{p_1}(n)=e^{i \theta (n)}$ for $n=n_1 \in \mathbb{Z}_N^2$, and zero otherwise,  
  where the real number $\theta(n)$ depends on $n.$ Therefore, we have  $\widehat{\varphi_{p_1}}(m)=
  \displaystyle \sum_{n \in \mathbb{Z}_N^2}\varphi_{p_1}(n) e^{-2\pi i \langle{m,n}\rangle/N}= e^{i \theta(n_1)}
  e^{-2\pi i \langle{m,n_1}\rangle/N}$, and hence $|\widehat{\varphi_{p_1}}(m)|=1$ for all $m \in \mathbb{Z}_N^2$. Thus, we  conclude that $\varphi_{p_1}$ is not frequency localized which implies that the system $\mathfrak{B}(\Phi)$ is  not  time-frequency localized.
  
 Now, for computing a lower bound of $R_0$ 
  (greater than zero), we consider an $N^2 \times N^2$ 
matrix $M=\left( \langle{T_{Am} \varphi_p, T_{Ak}e_{\alpha_j}}\rangle\right),$
where rows of the matrix are varying over 
$m \in \mathfrak{D}$ and $0\leq p\leq q-1$, and columns over $k \in \mathfrak{D}$ and $0\leq j \leq q-1$. Observe that, for each 
$k, m \in \mathfrak{D}$ and $0\leq p,  j\leq q-1$, we have
$$
\displaystyle \sum_{m \in \mathfrak{D}}\sum_{p=0}^{q-1}
{|\langle{T_{Am} \varphi_p, T_{Ak}e_{\alpha_j}}\rangle|}^2={||T_{Ak}e_{\alpha_j}||}^2=1, \  \mbox{and} \ \displaystyle \sum_{k \in \mathfrak{D}}\sum_{j=0}^{q-1}
{|\langle{T_{Ak}e_{\alpha_j}, T_{Am}\varphi_p}\rangle|}^2={||T_{Am}\varphi_p||}^2=1.
$$ 
Therefore,  rows and columns of $ {M}$ have unit norm. Further, for $0 \leq j_1 \neq j_2,p \leq q-1$ and  
$m, k_1 \neq k_2 \in \mathfrak{D}$, the inner product of any two columns of $ {M}$  given by 
$$
\displaystyle \sum_{m \in \mathfrak{D}}\sum_{p=0}^{q-1} \langle{T_{Am} \varphi_p, T_{A{k_1}}e_{j_1}}\rangle
\overline{\langle{T_{Am}\varphi_p, T_{A{k_2}}e_{j_2}}\rangle}=\langle{T_{A{k_2}}e_{j_2}, T_{A{k_1}}e_{j_1}}\rangle=0,$$
implies that columns of $ {M}$ are orthogonal. Similarly, we can check rows of $ {M}$ are orthogonal, and hence
$ {M}$ is an orthonormal matrix having the sum of squares of its entries equal to $N^2$. Therefore, all of its entries
cannot be less than $1/N$. This follows by noting that if we assume that all entries of $ {M}$ are less than $1/N$,  then, sum of 
squares of all $N^4$ entries of $ {M}$ will be less than $N^4\times \frac{1}{N^2}$, that is, $N^2$, which is not true. 
Further, observe that the absolute values  of all entries of matrix $ {M}$ are exactly the elements of the set $R_1$. Hence, we conclude 
that there exists $h \in R_1$ such that $h\geq 1/N$.  Thus, $R_0 \geq 1/N$. Hence, we have $1/N \leq R_0 < 1$.  

Similarly, we can show that $1/N \leq E_0 \leq 1$ by considering elements of  $\mathfrak{B}(\Phi)$ 
and  $\mathfrak{F}$.
\end{proof} 



\begin{thebibliography}{99}

\bibitem{DH} D. Donoho and  X. Huo, 
Uncertainty principles and ideal atomic decomposition,  
IEEE Trans. Inform. Theory,  47(7) (2001),  2845-2862. 

\bibitem{EB} M. Elad and  A.M. Bruckstein, A generalized uncertainty principle and sparse representation in pairs of bases,  IEEE Trans. Inform. Theory,  48 (2002), 2558--567. 


  \bibitem{F} M. Frazier, An introduction to wavelets through linear algebra,  Undergraduate Texts in
Mathematics, Springer-Verlag, New York, (1999),  xvi+501 pp.


 \bibitem{GJ} S. Ghobber and  P. Jaming,
On uncertainty principles in the finite dimensional setting, 
Linear Algebra Appl.,  435(4) (2011),  751-768.


\bibitem{GLT} S.S. Goh, S.L. Lee and  K.M. Teo, Multidimensional periodic multiwavelets, J. Approx. Theory,  98 (1999), 72--03. 

\bibitem{GT} S.S. Goh and  K.M. Teo, 
Wavelet frames and shift-invariant subspaces of periodic functions, 
Appl. Comput. Harmon. Anal., 20(3) (2006),  326-344.
 
\bibitem{H}   M. Holschneider, Wavelet analysis over abelian groups,
Appl. Comput. Harmon. Anal., 2(1) (1995), 52-60.

\bibitem{KPR} F. Krahmer, G.E. Pfander and  P. Rashkov, Uncertainty in time-frequency representations on finite abelian
groups and applications, Appl. Comput. Harmon. Anal.,  25(2) (2008),  209-225.

 

\bibitem{K} G. Kutyniok, Data separation by sparse representations,  Compressed sensing, Cambridge Univ. Press, 
Cambridge (2012), 485-514. 

\bibitem{PBV} R. Parhizkar, Y. Barbotin and M. Vetterli,
Sequences with minimal time-frequency uncertainty, 
Appl. Comput. Harmon. Anal.,  38(3) (2015),  452-468. 

\bibitem{RT} B. Ricaud and  B. Torr\'esani, A survey of uncertainty principles and some signal processing applications, 
Adv. Comput. Math.,  40(3) (2014),  629-650.

\bibitem{V} J.M. Vuletich, Orthonormal bases and filings of the time-frequency plane for music processing, 
Proc. SPIE 5207, Wavelets: Applications in Signal and Image Processing X (2003).

\bibitem{W} P. Wojtaszczyk, A mathematical introduction to wavelets, 
London Mathematical Society Student Texts-37, Cambridge University Press, Cambridge, (1997),  xii+261 pp.


\end{thebibliography}
\end{document}